\documentclass[11pt,oneside]{article}
\pdfoutput=1

\usepackage{amsmath,amsthm,amsfonts,amssymb,bbm,bm}
\usepackage{enumitem}

\usepackage[utf8]{inputenc}
\usepackage[english]{babel}

\usepackage[sort]{natbib} 
\usepackage{color}

\usepackage{tikz}
\usepackage{rotating}

\usepackage{hyperref}
\hypersetup{
   linktoc=page,
   colorlinks=true,
   linkcolor=blue,
   citecolor=blue
}

\usepackage[a4paper,top=1in, bottom=1.6in]{geometry}

\theoremstyle{definition}
\newtheorem{definition}{Definition}[section]
\newtheorem{example}[definition]{Example}

\theoremstyle{plain}
\newtheorem{theorem}[definition]{Theorem}
\newtheorem{prop}[definition]{Proposition}
\newtheorem{cor}[definition]{Corollary}
\newtheorem{lemma}[definition]{Lemma}

\theoremstyle{remark}

\newcommand{\cF}{\mathcal{F}}

\newcommand{\cE}{\mathcal{E}}
\newcommand{\cU}{\mathcal{U}}
\newcommand{\cV}{\mathcal{V}}
\newcommand{\cC}{\mathcal{C}}
\newcommand{\sA}{\mathsf{A}}
\newcommand{\sO}{\mathsf{O}}

\newcommand{\RR}{\mathbb{R}}

\newcommand{\real}{\mathbb{R}}
\newcommand{\eps}{\varepsilon}
\renewcommand{\natural}{\mathbb{N}}
\newcommand{\dd}{\, \mathrm{d}}

\newcommand{\one}[1]{\ensuremath{\mathbbm{1}}(#1)}

\DeclareMathOperator{\intr}{int}

\title{Why scoring functions cannot assess tail properties}

\author{Jonas Brehmer\footnote{Institute of Mathematics, University of Mannheim, 68131 Mannheim, Germany. Email: jbrehmer@mail.uni-mannheim.de} {} and Kirstin Strokorb\footnote{School of Mathematics, Cardiff University, Cardiff CF10 4AG, UK. Email: strokorbk@cardiff.ac.uk}}
\date{October 7, 2019}

\begin{document}

\maketitle

\begin{abstract}
  Motivated by the growing interest in sound forecast evaluation techniques with an emphasis on distribution tails rather than average behaviour, we investigate a fundamental question arising in this context: Can statistical features of distribution tails be elicitable, i.e.\ be the unique minimizer of an expected score? We demonstrate that expected scores are not suitable to distinguish genuine tail properties in a very strong sense. Specifically, we introduce the class of max-functionals, which contains key characteristics from extreme value theory, for instance the extreme value index. We show that its members fail to be elicitable and that their elicitation complexity is in fact infinite under mild regularity assumptions. Further we prove that, even if the information of a max-functional is reported via the entire distribution function, a proper scoring rule cannot separate max-functional values. These findings highlight the caution needed in forecast evaluation and statistical inference if relevant information is encoded by such functionals.
\end{abstract}

{\small
\noindent 
\textit{Keywords}:  Elicitability, elicitation complexity, extreme value index, max-functional, proper scoring rule, scoring function, strict consistency, tail equivalence.
\smallskip\\
  \noindent \textit{2010 MSC}: {Primary 62C05 (62C99); 62G32}\\ 
  \phantom{\textit{2010 MSC}:} {Secondary 91B06; 91B30} 
}

\section{Introduction}

Many of our day-to-day decisions rely on our ability to produce reasonable forecasts for quantities of interest.
For example, production planning involves forecasts on consumer demand, decisions in farming depend on information about the likely weather conditions and financial risk management uses statistical features of portfolio losses. Usually, such quantities are modelled via a random variable $Y$ having an unknown probability distribution and the reasonable actions of a decision maker depend on the properties of this distribution. Forecasts can encode such properties via real numbers, e.g.\ means or quantiles of the distribution, via sets, e.g.\ a confidence interval, or by a report of the whole distribution function.

When several competing forecasts are available, a crucial problem is to determine which one is most valuable. A principled approach to this task is to compare the forecasts to a set of realizations of $Y$ via a scoring rule or a scoring function, see e.g. \citet{GneitRaft2007} and \citet{Gneit2011}. A \textit{scoring function} assigns a real-valued score based on a forecast and a realizing observation. If a functional, i.e.\ a statistical property, of a distribution is the unique minimizer of the expected score with respect to this distribution, it is called \textit{elicitable}.
Elicitability is a desirable property for comparative forecast evaluation, where it can be used to incentivize risk-neutral forecasters to report their beliefs \citep{Gneit2011}. Moreover, elicitable functionals enable regression and M-estimation \citep{FissZieg2016,Gneit2011} and are central to various machine learning algorithms \citep{Steinetal2014,FronKash2018}. Recent theoretical advances on scoring functions and elicitability in the real-valued case can be found in \citet{Lambetal2008}, \citet{Gneit2011} and \citet{Steinetal2014}. More general vector-valued functionals are treated in \citet{FronKash2015b,FronKash2018} and \citet{FissZieg2016,FissZieg2019}.

Many statistical functionals such as expectations, quantiles, and expectiles are elicitable and there exist convenient characterizations of the corresponding classes of consistent scoring functions, cf.\ \citet{Gneit2011} and the references therein.  On the other hand, several widely considered functionals fail to be elicitable, for instance the variance, the mode \citep{Hein2014} and the prominent financial risk measure Expected Shortfall (ES) \citep{Weber2006,Gneit2011}. The non-elicitability of the latter functional can be addressed via more general notions of elicitability: \citet{FissZieg2016} show that ES is \textit{jointly elicitable} with the risk measure Value at Risk (VaR), where the latter is simply an extreme quantile. In other words, ES has \textit{elicitation complexity} equal to two in the sense of \citet{FronKash2018}. In this particular instance the elicitability problems associated with ES can be resolved, at the cost of considering a higher dimensional problem.

More generally, there is a recent growing interest in sound forecast evaluation techniques with an emphasis on distribution tails rather than average behaviour. 
For instance, \citet{FriedThor2012} investigate the use of scoring rules for distribution classes central to extreme value theory, and \citet{Diksetal2011}, \citet{Lerchetal2017} as well as \citet{HolzKlar2018} consider weighted scoring rules for forecasts of distribution tails.
An event-based approach to evaluate whether exceedances of high thresholds are predicted correctly is pursued by \citet{Stephetal2008} and \citet{FerroSteph2011}.
Closely connected is the verification tool of \citet{Tailletal2019} which is based on the asymptotic behavior of the continuously ranked probability score (CRPS), conditional on high realizations.
A fundamental question arising in this context is to what extent, and in which sense, statistical features of distribution tails are elicitable. The latter problem is the central theme of this manuscript.

In our approach to this question we introduce the concept of max-functionals which naturally arises from a key feature shared by the statistical functionals that are typically considered in extreme value theory. We demonstrate that max-functionals fail to be elicitable in a very strong sense. Consequently, it is natural to ask whether part of the problem can be mitigated by abandoning point forecasts in favor of reports of the entire distribution function. In this regard we generalize a result by \citet{Tailletal2019} and show that it is an inherent property of \emph{all} proper scoring rules that they cannot perfectly distinguish among different max-functional values.

The manuscript is organized as follows. In Section~\ref{sec:3concepts} we review the three notions of elicitability that are used in the recent literature. Section~\ref{sec:main} introduces the class of max-functionals and shows that they cannot be elicitable and that their elicitation complexity is infinite under mild assumptions. Section~\ref{sec:examples} provides examples of widely used max-functionals. In Section~\ref{sec:scoringrules} we turn to reports of entire distributions. We show that arbitrary large differences in tail behaviour, either quantified by tail equivalence or max-functionals, can remain undetected by proper scoring rules. Section~\ref{sec:discuss} concludes with a discussion of the results.

\section{Prerequisites: Elicitability and elicitation complexity}
\label{sec:3concepts}

For the reader's convenience this section recalls the central definitions of elicitability  and reviews basic findings. A more detailed overview of the existing literature is given in  \citet{FissZieg2016} and \citet{Gneit2011}, whose notation we follow here. Let $\sO \subseteq \real^d$ be a fixed set, called \textit{observation domain}, equipped with Borel $\sigma$-algebra $\mathcal{O}$. We use $\cF$ to denote a collection of probability distributions on $(\sO, \mathcal{O})$, whilst also identifying probability distributions with their cumulative distribution functions. A \textit{functional} will be a mapping $T: \cF \rightarrow \sA$ where $\sA \subseteq \real^n$ is called \textit{action domain}. A measurable function $g: \sO \rightarrow \real$ is called \textit{$\cF$-integrable} if it is integrable with respect to all $F \in \cF$. Analogously, a function $g: \sA \times \sO \rightarrow \real$ is called $\cF$-integrable if for all $x \in \sA$ the function $y \mapsto g(x,y)$ is integrable with respect to all $F \in \cF$. We use the short notation
\begin{align*}
\bar h (F) := \int_{\sO} h(y) \dd F(y) \quad \text{ and } \quad \bar{g} (x, F) := \int_{\sO} g(x,y) \dd F(y)
\end{align*}
for $\cF$-integrable functions $h,g$ and $x \in \sA$, $F \in \cF$.

\paragraph{Scoring functions and elicitability}

In the following, $S: \mathsf{A} \times \mathsf{O} \rightarrow \mathbb{R}$ denotes a \textit{scoring function}, i.e.\ an $\cF$-integrable function. The central concepts connecting scoring functions and statistical functionals are consistency and elicitability.

\begin{definition}[Consistency]
A scoring function $S: \mathsf{A} \times \mathsf{O} \rightarrow \mathbb{R}$ is $\mathcal{F}$-\textit{consistent} for a functional $T: \cF \rightarrow \sA $ if for all $x \in \mathsf{A}$ and $F \in \mathcal{F}$ we have $\bar{S}(x, F) \geq \bar{S}(T(F), F)$. It is called \textit{strictly} $\mathcal{F}$-\textit{consistent} for $T$ if it is $\mathcal{F}$-consistent for $T$ and for all $x \in \mathsf{A}$ and $F \in \mathcal{F}$ the equality $\bar{S}(x,F) = \bar{S}(T(F),F)$ implies $x = T(F)$.
\end{definition}

\begin{definition} [(Joint) elicitability] \label{def:jointlyelicitable}
A functional $T : \cF \rightarrow \sA \subseteq \real^n$ is called \textit{elicitable} if there exists a strictly $\cF$-consistent scoring function for $T$. It is called \textit{jointly elicitable with the functional} $T' : \cF \rightarrow \sA' \subseteq \real^k$ if $(T, T')$ is an elicitable functional.
\end{definition}

An important necessary condition that a statistical functional needs to satisfy in order to be elicitable is \textit{convexity of level sets}, which goes back to \citet{Osband1985}, cf.\ for instance \citet[Theorem~6]{Gneit2011} and \citet[Lemma~1]{Lambetal2008} for a proof.

\begin{theorem}[Convexity of level sets]  \label{th:convexlevel}
Let $T : \cF \rightarrow \sA$ be an elicitable functional. If $F_0, F_1 \in \cF$ and $\lambda \in (0,1)$ are such that $F_\lambda = \lambda F_1 + (1-\lambda) F_0 \in \cF$, then $T(F_0) = T(F_1) = t$ implies $T(F_\lambda) = t$.
\end{theorem}

\begin{example}  \label{ex:meanvar}
The simplest example of an elicitable functional is the mean of a distribution. More precisely, let $g: \sO \rightarrow \real$ be such that $g$ and $g^2$ are $\cF$-integrable and define $T: \cF \rightarrow \real$ via $T(F) = \bar g (F)$. Then $T$ is elicitable with a strictly $\cF$-consistent scoring function given by $S(x,y) = (x - g(y))^2$, the ubiquitous squared error loss. Likewise, the moment functionals defined via $T_k (F) := \int y^k \dd F(y)$ for $k \in \natural$ are elicitable.

A simple example of a non-elicitable functional is the variance functional $T_\mathrm{var}(F) := T_2(F) - T_1(F)^2$, whose non-elicitability follows directly from Theorem~\ref{th:convexlevel}. Nevertheless, $T_\mathrm{var}$ is jointly elicitable since the vector $(T_1, T_\mathrm{var} )$ can be obtained from the elicitable vector $(T_1, T_2)$ via a bijection and hence it is elicitable, see e.g.\ \citet[Theorem~4]{Gneit2011}. Another notable property is that on every subset of $\cF$ where $T_1$ is constant, $T_\mathrm{var}$ reduces to a shifted version of the second moment $T_2$ and is thus elicitable on this subset. That is, $T_\mathrm{var}$ is conditionally elicitable given $T_1$ in the following sense.
\end{example}

\begin{definition}[Conditional elicitability] \label{def:conditionale}
Let $T: \mathcal{F} \rightarrow \sA \subseteq \real^n$  and $T' : \cF \rightarrow \sA' \subseteq \real^k$ be functionals and let $T'$ be elicitable. For any $x \in \sA'$ define the set
\begin{align*} 
\mathcal{F}_{ x } := \lbrace F \in \mathcal{F} \mid T'(F) = x \rbrace .
\end{align*}
Then the functional $T$ is called \textit{conditionally elicitable given} $T'$ if for any $x \in \sA'$ its restriction to the class $\cF_x$ is elicitable.
\end{definition}

The concept of conditional elicitability was first introduced by \citet{Emmetal2015} and  motivated by a conditional backtesting approach for Expected Shortfall (ES) forecasts. A slight generalization was given by \citet{FissZieg2016}. Our definition coincides with the one from \citet{FissZieg2016} except that we drop the condition that $T'$ has elicitable components and only require it to be elicitable. This allows for a more convenient presentation of our results below.

Neither joint elicitability nor conditional elicitability imply elicitability, which follows from Example~\ref{ex:meanvar} with the variance functional serving as a counterexample. If a functional $T$ is jointly elicitable with the functional $T'$, and $T'$ is elicitable, then it is conditionally elicitable given $T'$. Conversely, as discussed in \citet{FissZieg2016}, it is unclear under which conditions a conditionally elicitable functional is jointly elicitable.

\paragraph{Elicitation complexity}

The definitions of joint elicitability and conditional elicitability both require a second elicitable functional $T'$ accompanying the functional of interest. The distinction between both functionals is made more explicit in the concept of \textit{elicitation complexity}. To illustrate this, recall Example~\ref{ex:meanvar} and note that the variance functional satisfies $T_\mathrm{var} = f(T_1, T_2)$, where $ f(x_1, x_2) = x_2 - x_1^2$. Since $T_1$ and $T_2$ are elicitable, we say that the variance functional has complexity 2. In general, $T$ has elicitation complexity at most $k$ if there is an elicitable functional $T' : \cF \rightarrow \sA' \subseteq \real^k$ such that $T = f(T')$ holds. Any $f$ and $T'$ satisfying this condition are then called \textit{link function} and \textit{intermediate functional}, respectively. The smallest dimension $k$ for which such a representation is feasible is the elicitation complexity.

\begin{definition}[Elicitation complexity]   \label{def:eliccomplex}
For any set of distribution functions $\cF$ the set of $\real^k$-valued elicitable functionals defined on $\cF$ is denoted via $\cE_k (\cF)$. For a functional $T: \mathcal{F} \rightarrow \sA \subseteq \real$ and sets $\cC_k \subseteq \cE_k (\cF)$ the \textit{elicitation complexity of $T$ with respect to} $(\cC_k)_{k \in \natural}$ is defined via
\begin{align*}
\mathsf{elic} (T) := \min \lbrace k \in \natural \mid \exists \, T' \in \cC_k : \, T = f \circ T' \text{ for some } f: T'(\cF) \rightarrow \sA \rbrace .
\end{align*}
If the minimum is not attained for any $k \in \natural$, the elicitation complexity of $T$ with respect to $(\cC_k)_{k \in \natural}$ is infinite and we write $\mathsf{elic}(T) = \infty$. 
\end{definition}

Elicitation complexity was introduced by \citet{Lambetal2008} and further analyzed in \citet{FronKash2018}, the latter motivated by its role in empirical risk minimization (ERM) algorithms in machine learning. Intuitively speaking, it replaces the question \emph{whether} a functional is elicitable by the question
\textit{how complex} it is to elicit the functional.

If no regularity conditions are imposed on $f$ or $T'$, this can lead to small complexities without clear benefits in applications. More precisely, if $f$ is arbitrary and  $\cC_k = \cE_k (\cF)$ is chosen, pathological choices of $f$, like bijections from $\real^k$ to $\real$, cause all functionals to have complexity 1, as demonstrated by \citet[Remark~2]{FronKash2018}. To avoid such problems, it is standard to choose suitable subclasses $\cC_k$ of intermediate functionals. One possible choice, which is used by \citet{FronKash2018} as well as \citet{DearFron2019}, is $\cC_k := \mathcal{I}_k (\cF) \cap \cE_k (\cF)$, where $\mathcal{I}_k (\cF)$ is the set of $\real^k$-valued identifiable functionals on $\cF$. Another possibility, implicitly used by \citet{Lambetal2008} is to define $\cC_k$ to be a subclass of all functionals which have elicitable components.

Lastly, it is also possible to impose regularity on the link function $f$, e.g.\ by requiring differentiability or continuity. Notably, joint elicitability can be understood as a version of elicitation complexity where the link function is the projection on the last component \citep{FronKash2018}.

We need to be cautious when interpreting elicitation complexity, since imposing different regularity conditions via $(\cC_k)_{k \in \natural}$ can lead to different elicitation complexities for the same functional, see \citet[Subsection~2.2]{FronKash2018} for an example. In particular, some $\real^k$-valued functional might be elicitable and simultaneously have elicitation complexity strictly greater than $k$. Conversely, a functional can have elicitation complexity 1, although it is not itself elicitable, as illustrated in \citet[Remark~1]{FronKash2018}.

We conclude this section with a lemma which considers the properties of a functional $T$ if it is restricted to some subclass $\cF_2 \subseteq \cF$. The first statement corresponds to the first part of Lemma~2.11 of~\citet{FissZieg2015}, the second and third statement are simple extensions. Their proofs are straightforward and therefore omitted.

\begin{lemma}  \label{lem:SmallerSets}
Let $T: \cF \rightarrow \sA$ be a functional and let $\cF_2 \subseteq \cF$ be non-empty.
\begin{enumerate}[label=(\alph*)]
\item If $T$ is elicitable, then the restricted functional $T_{\vert \cF_2}$ is elicitable.

\item If $\mathsf{elic}(T) = k$ with respect to $(\cC_k)_{k \in \natural}$ and we define $\cC_k^2 := \lbrace T'_{\vert \cF_2} \mid T' \in \cC_k \rbrace$, then $\mathsf{elic}(T_{\vert \cF_2} ) \leq k$ with respect to $(\cC_k^2)_{k \in \natural}$.

\item If $\mathsf{elic}(T) = k$ with respect to $(\cC_k)_{k \in \natural}$ and sets $(\cC_k')_{k \in \natural}$ satisfy $\cC_k \subseteq \cC_k'$ for all $k \in \natural$, then $\mathsf{elic}(T) \leq k$ with respect to $(\cC_k')_{k \in \natural}$.
\end{enumerate}
\end{lemma}

\section{The elicitation complexity of max-functionals} \label{sec:main}

This section introduces max-functionals, the central objects of our study, and investigates their elicitability as well as their elicitation complexity. Henceforth, let $\cF$ always denote a \textit{convex} class of distributions.

\begin{definition}
  A functional $T: \cF \rightarrow \real$ is called \textit{max-functional} if
  \begin{align*}
  T(\lambda F_1 + (1-\lambda) F_0) = \max (T(F_0), T(F_1) )
  \end{align*}
  holds for all $F_0, F_1 \in \cF$ and $\lambda \in (0,1)$.
\end{definition}

The essential feature of a max-functional is that its value on convex combinations of distributions is determined by the values attained on the extreme points. Equivalently, we can also define min-functionals and all results carry over with minor modifications. The constant functional is the simplest max-functional, but we will usually not be interested in this trivial case. Instead, Section~\ref{sec:examples} collects some non-trivial examples of max-functionals that are routinely considered in extreme value theory. Also note that, by definition, restrictions of max-functionals to a certain set of values are again max-functionals.

\begin{lemma}
Let $T: \cF \to \RR$ be a max-functional and $A \subset \RR$ a set. Set $\cF_A:=\lbrace F \in \cF \mid T(F) \in A \rbrace$, then $\cF_A$ is convex and the restricted functional $T: \cF_A \to A \subset \RR$ is also a max-functional.
\end{lemma}

\paragraph{Non-elicitability of max-functionals}

We start by proving that max-func\-tionals cannot be elicitable. As remarked in Section~\ref{sec:3concepts} the usual way to show that a functional is not elicitable consists of applying Theorem~\ref{th:convexlevel}, i.e.\ showing that it fails to have convex level sets. However, any max-functional has convex level sets by definition. So this approach is not feasible, as in the case of the mode functional \citep{Hein2014}. Instead, we employ the following new criterion.

\begin{theorem} \label{th:Criterion2}
  Let $T: \cF \rightarrow \sA$ be a functional. If there are $F_0, F_1 \in \mathcal{F}$ such that $T(F_0) \neq T(F_1)$ and
  \begin{align*}
    T( \lambda F_1 + (1-\lambda) F_0 ) \in \lbrace T(F_0), T(F_1) \rbrace
    \quad \text{for all } \lambda \in (0,1),
  \end{align*}
  then $T$ is not elicitable. 
\end{theorem}

\begin{proof}
Set $x_0 := T(F_0)$, $x_1 := T(F_1)$ and $F_\lambda := \lambda F_1 + (1-\lambda) F_0 $ and let $x_0 \neq x_1$. Assume that $S$ is a strictly consistent scoring function for $T$. Then we have
\begin{align*}
\bar{S}(x_0, F_\lambda) - \bar{S}(x_1, F_\lambda) &= \lambda (\bar{S}(x_0, F_1) - \bar{S}(x_1, F_1)) \\ 
&\phantom{=} \,+ (1- \lambda) ( \bar{S}(x_0, F_0) - \bar{S}(x_1, F_0) )
\end{align*}
and the first difference $\bar{S}(x_0, F_1) - \bar{S}(x_1, F_1)$ is positive, while the second difference $ \bar{S}(x_0, F_0) - \bar{S}(x_1, F_0)$ is negative. Consequently, $\bar{S}(x_0, F_\lambda) = \bar{S}(x_1, F_\lambda)$ for some $\lambda \in (0,1)$. Since either $T(F_\lambda) = x_0$ or $T(F_\lambda) = x_1$ holds by assumption, we arrive at a contradiction. 
\end{proof}

\begin{cor}  \label{cor:maxnotelic}
If $T: \cF \rightarrow \real$ is a non-constant max-functional, then it is not elicitable.
\end{cor}

Loosely speaking, Theorem~\ref{th:Criterion2} states that elicitable functionals cannot be piecewise constant on convex combinations of distributions. It is closely connected to Theorem~\ref{th:convexlevel}, but of independent interest beyond its use to establish non-elicitability for max-functionals.
\citet{Fissetal2019} use similar arguments as in the proof of Theorem~\ref{th:Criterion2} to study necessary conditions for the level sets of $T$ in the context of set-valued functionals $T : \cF \to 2^\sA$, where $2^\sA$ denotes the power set of~$\sA$. Apart from that
\citet{FronKash2018} state that `\textit{no nonconstant finite-valued property is identifiable}'. Theorem~\ref{th:Criterion2} implies the following analogon.

\begin{cor}
If $T: \cF \rightarrow \real$ is a non-constant finite-valued functional, then it is not elicitable.
\end{cor}

\paragraph{Elicitation complexity of max-functionals}

Turning from the elicitability question to the elicitation complexity of max-functionals, the question of elicitation complexity is only meaningful in relation to a family of sets $(\cC_k)_{k \in \natural}$, where each set $\cC_k \subset \cE_k (\cF)$ is a collection of reasonably regular $\real^k$-valued elicitable functionals, cf.\ Section~\ref{sec:3concepts}. Our major regularity requirement is \textit{mixture-continuity} as in \citet{BellBign2015} and \citet{FissZieg2019}.

\begin{definition}  \label{def:mixconti}
A functional $T: \cF \rightarrow \sA$ is called \textit{mixture-continuous} if for all $F_0, F_1 \in \cF$ such that $\lambda F_1 + (1-\lambda) F_0  \in \cF$ for all $\lambda \in [0,1]$, the mapping
\begin{align*}
[0,1] \rightarrow \sA, \quad \lambda \mapsto T( \lambda F_1 + (1-\lambda) F_0 )
\end{align*}
is a continuous function.
\end{definition}

Many statistical properties are mixture-continuous, e.g.\ ratios of expectations, quantiles and expectiles, see \citet{FissZieg2019} for details. \citet{Lambetal2008} consider only continuous functionals and \citet{FissZieg2019} and \citet{BellBign2015} show that under weak assumptions, an elicitable functional $T'$ is mixture-continuous if its expected score function $x \mapsto \bar S (x,F)$ is continuous for all $F \in \cF$. Therefore, a functional which is not mixture-continuous can have discontinuous expected scores, leading to difficulties in forecast evaluation, estimation and regression.

To avoid further degenerate behaviour, we impose a richness assumption on potential intermediate functionals $T'$ in the sense that we require the image $T'(\cF) \subseteq \real^k$ to have at least non-empty interior. This assumption is natural for large enough classes $\cF$ and was, for instance, used by \citet{FissZieg2016,FissZieg2019} when establishing results on consistent scoring functions for $T'$.

In addition to mixture continuity, we follow \citet{Lambetal2008} and consider only functionals with elicitable components. Summarising, the first family of functionals which we consider in our complexity result is
\begin{align*}
\cU_k := \left\lbrace T' \in \cE_k(\cF) \, \left|
\begin{array}{lc}
	T' \text{ mixture-continuous with elicitable} \\
	\text{components, } \intr (T' (\cF)) \neq \emptyset
\end{array} \right\rbrace \right. ,
\end{align*}
where $\intr (B)$ denotes the interior of a set $B \subseteq \real^k$. Alternatively, we require that the image $T'(\cF)$ of a potential intermediate functional $T'$ has not only non-empty interior, but is itself an open set, i.e.\ we consider the family
\begin{align*}
\cV_k := \left\lbrace T' \in \cE_k(\cF) \, \left|
\begin{array}{lc}
	T' \text{ mixture-continuous with elicitable} \\
	\text{components, } T' (\cF) \text{ open} 
\end{array} \right\rbrace \right.  .
\end{align*}
We are now in position to consider the elicitation complexity of max-functionals with respect to these families.

\begin{theorem} \label{th:complexity}
Let $T:\cF \rightarrow \real$ be a max-functional. Then the following hold true.
  \begin{enumerate}[label=(\alph*)]
  \item $T$ has elicitation complexity $\infty$ with respect to $(\cU_k)_{k \in \natural}$ unless $T(\cF)$ contains its supremum.
  \item $T$ has elicitation complexity $\infty$ with respect to $(\cV_k)_{k \in \natural}$ unless $T$ is constant.
  \end{enumerate}
\end{theorem}

\begin{proof}
Assume there is a $k \in \natural$, a surjective functional $T' : \cF \rightarrow \sA'$ in $\cU_k$ or $\cV_k$ and a function $f: \sA' \rightarrow \real$ such that $T = f \circ T'$. Without loss of generality, $T'$ is surjective, hence its mixture-continuity together with the assumed convexity of $\cF$ imply that $\sA'$ is path-connected. Since it has non-empty interior, we can choose a hyperrectangle $Q := \prod_{i=1}^{k} [ c_i, d_i ]  \subseteq \intr (\sA')$ and consider each component of $T'$ isolated on $Q$. To do so, choose a component $j \in \lbrace 1, \ldots, k \rbrace$ and a $z_i \in [c_i, d_i]$ for all $i \in \lbrace 1, \ldots, k \rbrace \backslash \lbrace j \rbrace$. We can then obtain $F_{c_j,z}, F_{d_j, z} \in \cF$ such that
\begin{align*}
T'(F_{c_j,z}) &= (z_1, \ldots, z_{j-1}, c_j, z_{j+1}, \ldots, z_k) \quad \text{ and}\\
T'(F_{d_j,z}) &= (z_1, \ldots, z_{j-1}, d_j, z_{j+1}, \ldots, z_k) .
\end{align*}
All components of $T'$ are elicitable and thus have convex level sets by Theorem~\ref{th:convexlevel}. Consequently, the $i$-th component, where $i \in \lbrace 1, \ldots, k \rbrace \backslash \lbrace j \rbrace$, equals $z_i$ for all convex combinations of $F_{c_j,z}$ and $F_{d_j,z}$. If we define
\begin{align*}
\sA_{j,z}' :=  \lbrace (z_1, \ldots, z_{j-1}, x, z_{j+1}, \ldots, z_k)  \mid x \in (c_j, d_j) \rbrace  \subseteq Q ,
\end{align*}
the fact that the $j$-th component has convex level sets and is mixture-continuous implies that for all $a \in \sA_{j,z}'$ there exists a $\lambda \in (0,1)$ with  $T' ( \lambda F_{c_j,z} + (1-\lambda)F_{d_j,z} ) = a$. The connection $T = f \circ T'$ now gives
\begin{align*}
f( (z_1, \ldots, z_{j-1}, x, z_{j+1}, \ldots, z_k) ) &= f( T' ( \lambda F_{c_j,z} + (1-\lambda) F_{d_j,z} ) \\
&= T( \lambda F_{c_j,z} + (1-\lambda) F_{d_j,z}) \\
&= \max ( T(F_{c_j,z}) , T(F_{d_j,z}) ),
\end{align*}
for all $x \in (c_j, d_j)$, implying that $f$ has to be constant on the set $\sA_{j,z}'$. Repeating this argument for any choice of $j \in \lbrace 1, \ldots, k \rbrace$ and $z_i \in [c_i, d_i]$ with $i \in \lbrace 1, \ldots, k \rbrace \backslash \lbrace j \rbrace$ shows that there is a $C \in \real$ such that $f(q) = C $ for all $q \in \intr ( Q)$.

Now fix $x_0 \in \intr (Q)$. For any $x_1 \in \sA'$ we can choose distributions $F_0, F_1 \in \cF$ with $T' (F_0)= x_0$ and $T' (F_1 ) = x_1$. Since $x_0 \in \intr (Q)$  and $T'$ is mixture-continuous, there is a small $\mu \in (0,1)$ such that $T'(\mu  F_1 + (1-\mu)F_0 ) \in \intr (Q)$ holds. We thus obtain
\begin{align*}
C  = f(T'( \mu F_1 + (1- \mu) F_0 ) ) &= T( \mu F_1 + (1- \mu) F_0 ) \\
&= \max (T(F_0) , T(F_1) ) \\
&= \max ( f(x_0), f(x_1) ) = \max (C , f(x_1) ),
\end{align*}
implying $f(x_1) \leq C$. Since $x_1$ was arbitrary, we have $f(x) \leq C$ for all $x \in \sA'$, showing $C = \sup T(\cF)$ and proving statement~(a).

Assume now that $\sA'$ is open. Then for every $x_1 \in \sA'$ there is a hyperrectangle $Q_1 \subseteq \sA'$ such that $x_1 \in \intr (Q_1)$. Arguing as in the beginning of the proof gives $f(q) = f(x_1)$ for all $q \in \intr (Q_1)$. So letting $T'(F_1) = x_1$ as above we obtain a $\nu \in (0,1)$ such that $T'(\nu F_1 + (1-\nu) F_0) \in \intr (Q_1)$. This implies 
\begin{align*}
C = f(T'(\mu F_1 + (1-\mu) F_0) ) &= \max ( T(F_0), T(F_1)) \\
&= f(T' (\nu F_1 + (1-\nu) F_0) ) = f(x_1) .
\end{align*}
Since $x_1$ was arbitrary, $T$ must be constant, proving part~(b).
\end{proof}

Theorem~\ref{th:complexity} implies infinite elicitation complexity of max-functionals in a wide range of natural settings. Ultimately, our main interest lies in understanding the elicitation complexity with respect to the more general family $\cU_k$, which imposes only very weak assumptions on a potential intermediate functionals.

\begin{cor} \label{cor:complexity}  
Let $T:\cF \to \real$ be a max-functional and let one of the following conditions be satisfied.
\begin{enumerate}[label=(\roman*)]
	\item $T$ is unbounded.
	\item $T$ is surjective onto an open interval $(a,b)$.
	\item $T$ is surjective onto a half-open interval $[a,b)$.
\end{enumerate}
Then $T$ has elicitation complexity $\infty$ with respect to $(\cU_k)_{k \in \natural}$.
\end{cor}

Alternatively, considering elicitation complexity with respect to the family $(\cV_k)_{k \in \natural}$ amounts to requiring more regularity for a potential intermediate functional $T'$ and, in this case, \emph{all} non-constant max-functionals have infinite elicitation complexity. Lemma~\ref{lem:SmallerSets} further implies that the infinite elicitation complexity of max-functionals also extends to larger classes than the considered convex family of distribution functions $\cF$ and is valid with respect to smaller families contained in $(\cU_k)_{k \in \natural}$ or $(\cV_k)_{k \in \natural}$.

Finally, by definition, any functional of finite elicitation complexity is conditionally elicitable, but it is unclear whether the reverse implication holds. We thus conclude with showing that max-functionals with infinite elicitation complexity can neither be conditionally elicitable nor jointly elicitable.

\begin{theorem} 
Let $T : \cF \to \real$ be a max-functional such that $\mathsf{elic}(T) = \infty$ with respect to a family $(\cC_k)_{k \in \natural}$. Let $T' : \cF \to \sA'$ be a functional with $T' \in \cC_m$ for some $m \in \natural$. Then the following hold true.
\begin{enumerate}[label=(\alph*)]
	\item $T$ is not conditionally elicitable given $T'$.
	\item $T$ is not jointly elicitable with $T'$.
\end{enumerate}
\end{theorem}

\begin{proof}
For the first part assume conversely, that there is an $m \in \natural$ and a functional $T' \in \cC_m$ such that $T$ is conditionally elicitable given $T'$. That is, $T$ is elicitable on the subclass $\cF_x = \{ F \in \cF \mid T'(F) = x \} $ for any $x \in \sA'$. By assumption, there is no link function $f$ such that $T = f \circ T'$ holds. Consequently, there is at least one $z \in \sA' \subseteq \real^m$ such that $T$ is not constant on $\cF_{z}$. If $z$ defines such a class, then it is convex due to the elicitability of $T'$ and moreover we can find  $F_0, F_1 \in \cF_{z}$ such that $T(F_0) \neq T(F_1)$ holds. Theorem~\ref{th:Criterion2} now implies that the restriction of $T$ to $\cF_{z}$ cannot be elicitable, a contradiction to the conditional elicitability of $T$.

For the second part note that, as remarked in Section~\ref{sec:3concepts} and in the discussion of \citet{FissZieg2016}, the joint elicitability of $T$ with an elicitable functional $T'$ implies that $T$ is conditionally elicitable given $T'$. Consequently, the first part of the proof implies the result.
\end{proof}

We conclude this section with a technical remark. In the spirit of \citet{FronKash2018}, our complexity result (Theorem~\ref{th:complexity}) employs regularity assumptions on the possible intermediate functionals. The main assumption is that they possess elicitable components. Why this is essential is illustrated by the use of the hyperrectangle $Q$ in the proof. Intuitively, this assumption can be relaxed at the cost of more technical arguments. The main challenge hereby is to control the values of $T'$ in a small hyperrectangle (or ball) around some $x_0 \in \intr(\sA')$. However, we did not pursue this approach further, since we believe that our setting covers many functionals of practical interest and at the same time illustrates the irregular behaviour that will be inherent to any link function for a max-functional.

\section{Examples of max-functionals}   \label{sec:examples}

Prominent examples of max-functionals, to which the results of Section~\ref{sec:main} apply, are routinely considered in extreme value theory and are key characteristics for the purpose of inference on the tail of a distribution.

\paragraph{Upper endpoint}

For a real-valued random variable with distribution function $F$, its upper endpoint is the supremum of its support
\begin{align*}
x^F := \sup \lbrace x \in \real \mid F(x) < 1 \rbrace .
\end{align*}
By definition, the upper endpoint can be interpreted as a real-valued max-functional on the convex class $\lbrace F \in \cF \mid x^F < \infty \rbrace$. \citet[Example 3.9]{BellBign2015} discuss the upper endpoint under the name \textit{worst-case risk measure} and show that it is not elicitable, once further regularity conditions on the admissible scoring functions are imposed. In light of Corollary~\ref{cor:maxnotelic} the non-elicitability of the upper endpoint follows without any further assumptions. In addition it has infinite elicitation complexity in the sense of Theorem~\ref{th:complexity} and Corollary~\ref{cor:complexity}.

\paragraph{Index of regular variation / Tail index}
When the upper endpoint is infinite, another key characteristic to describe the tail behaviour of heavy-tailed distributions is the index of regular variation. A strictly positive measurable function $f$ satisfying
\begin{align*}
\lim_{x \rightarrow \infty} \frac{f(xt)}{f(x)} = t^\rho
\end{align*}
for $t > 0$ is called \textit{regularly varying (at infinity) with index} $\rho (f) \in \real$. For a distribution $F$ its index of regular variation is the respective index for its survival function $\overline{F} := 1- F$, that is, $T(F) := \rho ( \overline{F})$. Its inverse $T(F)^{-1}$ is also called \textit{tail index} in the risk management literature, cf.\ \citet[Section 5.1]{McNeiletal2015}.
If the tail $\overline{F}$ is regularly varying with (a negative) index $\rho$, this means that $\overline{F}$ decays essentially like a power function with decay rate $1/\rho$. Since $\rho (f + g) = \max ( \rho (f), \rho (g) )$ (cf.\ e.g.\ \citet[Proposition~B.1.9]{FerrHaan2006}), the index of regular variation $T$ is naturally a max-functional, while the tail index $T^{-1}$ is a min-functional.

\paragraph{Tail-separating functionals}
More generally, we can deduce that the property of \emph{`being a max-functional' (or min-functional)} is in fact inherent to all \emph{`tail-ordering indices'}. To make this precise, let us consider the following natural order on distribution tails. For two distribution functions $F$ and $G$ with upper endpoints $x^F, x^G \in \real \cup \{\infty\}$ we say that $G$ \textit{has heavier tail than} $F$ and write $F <_t G$ if
\begin{align*}
\text{either } \quad x^F < x^G \qquad \text{or} \qquad x^F = x^G = x^* \text{ and } \lim_{x \to x^*} \frac{ \overline{F} (x)}{ \overline{G} (x)} = 0 .
\end{align*}
We say that $F$ and $G$ are \textit{tail equivalent} and write $F \sim_t G$ if they share the same upper endpoint $x^F=x^G=x^* \in \real \cup \lbrace \infty \rbrace$ and
\begin{align*}
\underset{x \rightarrow x^*}{\lim} \frac{\overline{F}(x)}{\overline{G}(x)} \in (0, \infty).
\end{align*}
Note that ``$<_t$'' defines a strict partial order on any set of distribution functions $\cF$ and that for tail equivalent $F$ and $G$  neither $F <_t G$ nor $G <_t F$ can hold. The following proposition shows that a functional which respects the tail order  ``$<_t$'' is a max-functional.

\begin{prop}  \label{prop:tailordering}
Let $T: \cF \to \real$ be a functional that satisfies for all $F,G \in \cF$
\begin{align*}
T(F) - T(G)  \left\lbrace
\begin{array}{ll}
\leq 0 & \text{if } F <_t G, \\
\geq 0 & \text{if } G <_t F, \\
= 0 & \text{else}.
\end{array}
\right.
\end{align*}
Then $T$ is a max-functional.
\end{prop}

\begin{proof}
Let $F_0, F_1 \in \cF$ and set $F_\lambda := \lambda F_1 + (1-\lambda) F_0$ for $\lambda \in (0,1)$. We distinguish three cases. If $F_0 <_t F_1$, we have $x^{F_\lambda} = x^{F_1} \geq x^{F_0}$ and the identity
\begin{align*}  
\frac{\overline{F}_\lambda (x)}{\overline{F}_1 (x)} =  \lambda + (1-\lambda) \frac{\overline{F}_{0} (x)}{\overline{F}_1 (x)}
\end{align*}
for $x < x^{F_1}$ implies $F_\lambda \sim_t F_1$. Hence, neither $F_\lambda <_t F_1$ nor $F_1 <_t F_\lambda$ can be true. Together with $T(F_0) \leq T(F_1)$ we may conclude $T(F_\lambda) = T(F_1) = \max ( T(F_0), T(F_1) )$. By symmetry, the case $F_1 <_t F_0$ can be treated analogously. In the remaining case we have neither $F_0 <_t F_1$ nor $F_1 <_t F_0$, so $x^{F_1} = x^{F_0} = x^{F_\lambda} =: x^*$ must hold. Consequently, 
\begin{align*}
\liminf_{x \to x^*} \frac{\overline{F}_\lambda (x)}{\overline{F}_1 (x)} \geq \lambda > 0 \quad \text{ and } \quad \limsup_{x \to x^*} \frac{\overline{F}_\lambda (x)}{\overline{F}_1 (x)} < \infty,
\end{align*}
where the latter follows as the tail of $F_0$ is not heavier than the tail of $F_1$. This implies that neither $F_1 <_t F_\lambda$ nor $F_\lambda <_t F_1$ can hold true, which gives $T(F_\lambda) = T(F_1) = \max (T(F_0) , T(F_1))$ and concludes the proof.
\end{proof}

Another instance of a tail-ordering functional in the sense of Proposition~\ref{prop:tailordering} is the \emph{$\mathcal{M}$-index} as introduced in \cite{Cadena2016}. If it exists, it is the unique $\rho \in \RR$ such that 
\begin{align*} 
\lim_{x \to \infty} \frac{\overline F (x)}{x^{\rho + \eps}} = 0 \quad \text{and} \quad  \lim_{x \to \infty} \frac{\overline F (x)}{x^{\rho - \eps}} = \infty  \quad \text{ for all } \eps > 0 .
\end{align*}
It is easily seen that the $\mathcal{M}$-index coincides with the index of regular variation for distribution functions $F$ with regularly varying tail function $\overline F$. As it sorts survival functions according to their power law decay,  Proposition~\ref{prop:tailordering} implies that the $\mathcal{M}$-index is a max-functional.

\paragraph{Extreme value index}

A central characteristics of extreme value theory is the extreme value index, which classifies the limiting behaviour of rescaled maxima of growing samples from a distribution. More precisely, if there exist suitable location-scale normings $a_n>0$, $b_n\in \real$ such that the distribution functions $ F_n(x) := F^n(a_n x + b_n)$ converge weakly to a non-degenerate distribution function $G$, the limiting distribution function $G$ is necessarily a \emph{Generalized Extreme Value Distribution (GEV)}. This means that up to a location-scale normalization we have
\begin{align*}
  G(x)=G_\gamma(x)= \exp\{-(1+\gamma x)_+^{-1/\gamma}\}
\end{align*}
for some $\gamma=\gamma(F) \in \RR$, where $G_0(x)=\exp\{-e^{-x}\}$ for $\gamma=0$. The distribution $F$ is said to be in the \emph{max-domain of attraction} of $G=G_\gamma$ and the shape parameter $\gamma(F)$ is the \emph{extreme value index (EVI) of $F$}, cf.\ e.g.\ the monographs \citet{Resnick1987} and \citet{FerrHaan2006} for further background.

Let $\cF$ be the class of distribution functions which are in a max-domain of attraction for some GEV and consider first the EVI on the subclass of heavy-tailed distributions $\cF_+ = \lbrace F \in \cF \mid \gamma(F) > 0 \rbrace$. It is well-known that a distribution $F \in \cF$ has EVI $\gamma >0$ if and only if $\rho (\overline{F}) = -\gamma^{-1}$, where $\rho$ is the index of regular variation (cf.\ e.g.\ \citet[Proposition~1.11]{Resnick1987}). Consequently, the EVI $\gamma$ is also a max-functional on $\cF_+$.

When considering the class of light-tailed distributions, i.e.\ the case $\gamma(F) < 0 $, we need to specify an upper endpoint first in order to make `being a max/min-functional' meaningful for the EVI $\gamma$. To this end, let $\cF_{x^*} = \lbrace F \in \cF \mid \gamma(F) < 0, \, x^F=x^* \rbrace$. Again the EVI behaviour is governed by regular variation, since $\gamma(F)=-\gamma(F_*)$ with $F_*(x)=F(x^* - x^{-1})$ (cf.\ e.g.\ \citet[Proposition~1.13]{Resnick1987}). This shows that the EVI $\gamma$ is a min-functional on the class $\cF_{x^*}$. Note that it is crucial to assume equal upper endpoints, because otherwise it is not the EVI that dominates the tail behaviour, but the upper endpoint itself. \\

So far, we have looked at statistical indices that classify \emph{univariate} tail behaviour. However, similar issues arise when we want to quantify \emph{joint} tail behaviour in higher dimensions. Exemplary, let us consider the coefficient of tail dependence.

\paragraph{Coefficient of tail dependence}  \label{par:taildep}

In order to quantify the tail behaviour of a bivariate distribution function \citet{LedTawn1996,LedTawn1997} introduced the coefficient of tail dependence. For a bivariate distribution function $F$ of a random vector $(X_1, X_2)$ let us write $\overline F_i( x) := \mathbb{P} (X_i > x)$, $i=1,2$ and $\overline F( x) := \mathbb{P} (X_1 > x, X_2 > x)$ for the associated survival functions. Suppose there is an $\alpha > 0$ such that both $\overline{F}_1$ and $\overline{F}_2$ are regularly varying with index $-\alpha$. If in addition the joint survival function $\overline F$ is regularly varying with index $-\alpha/\eta$ for some $\eta \in (0,1]$, the coefficient $\eta=\eta(F)$ is called \textit{coefficient of tail dependence (CTD)} of the bivariate distribution $F$. Let us consider the CTD $\eta$ on the class of bivariate distributions
\begin{align*} 
\cF_\alpha  = \lbrace F \mid \rho(\overline{F}_1)=\rho(\overline{F}_2) = -\alpha, \, \rho(\overline{F})=-\alpha/\eta \text{ for some } \eta \in (0,1] \rbrace.
\end{align*}
Then it follows for $F,G \in \cF_\alpha$ that $\rho ( \lambda \overline F + (1-\lambda)\overline G) = - \alpha/\max(\eta(F),\eta(G))$  by the properties of the index of regular variation. Hence $\eta$ is a max-functional on $\cF_\alpha$.

\section{Proper scoring rules and max-functionals}  \label{sec:scoringrules}

In probabilistic forecasting, the whole distribution function instead of a single value is reported to the decision maker. Analogously to a scoring function, a \textit{scoring rule} then assigns a score based on the forecasted distribution and a realizing observation. The scoring rule is called \textit{proper} if its expected score with respect to a distribution is minimized whenever the forecast coincides with this distribution, see e.g.\ \citet{GneitRaft2007} or \citet{Dawid2007} for recent reviews.

In light of the results of Section~\ref{sec:main}, the following approach may seem reasonable to someone seeking information about a max-functional: Instead of single values, distribution functions are reported and evaluated via proper scoring rules. Then the max-functionals are computed from the forecasted distributions.

If the max-functional of interest is a property of the tail, e.g.\ the extreme value index, one could expect this method to work well as long as the scoring rule shows a good performance in the tails. In order to emphasize specific regions of interests, in particular the tails, \citet{GneitRanjan2011} and \citet{Diksetal2011} combined scoring rules with weight functions. Drawbacks and benefits of these weighted proper scoring rules were further studied in \citet{Lerchetal2017} and \citet{HolzKlar2018}, where the latter propose general construction principles. A theoretical problem is pointed out by \citet{Tailletal2019}, who show that weighted versions of the continuously ranked probability score (CRPS) cannot detect that two distributions are not tail equivalent.

This section shows that the problems detected by \citet{Tailletal2019} occur also for max-functionals and do not depend on the specific choice of proper scoring rule. Simply put, the expected score difference of two distributions can be arbitrarily small while their values for a max-functional can be large. As previously, $\cF$ is a convex set of distribution functions on $\sO \subseteq \real^d$. In our notation we follow \citet{GneitRaft2007} as well as Section~\ref{sec:examples}.

\begin{definition} [Scoring rule]
A real-valued function $S: \cF \times \sO \rightarrow \real$ is called \textit{scoring rule} if for all $F \in \cF$ the mapping $y \mapsto S(F,y)$ is  $\cF$-integrable. The scoring rule $S$ is called \textit{proper} if $\bar{S}(F,F) \leq \bar{S}(G,F)$ holds for all $F,G \in \cF$. It is \textit{strictly proper} if it is proper and for any $F,G \in \cF$ the equality $\bar{S}(G,G) = \bar{S}(F,G)$ implies $G = F$.
\end{definition}

For clarity of presentation we require all scoring rules to be $\cF$-integrable, while \citet{GneitRaft2007} only require \textit{quasi-integrability}. The latter means that the expected score $\bar{S}(G,F)$ is well-defined (and not necessarily finite) for all $G,F \in \cF$. Our assumption of $\cF$-integrability is however only a minor restriction, which can be relaxed as discussed below.

A popular choice of scoring rule is the \emph{(weighted) continuous ranked probability score}, abbreviated by \emph{CRPS (wCRPS)}. For some weight function $w: \real \rightarrow [0, \infty)$ the wCRPS is defined via
\begin{align*}
\text{wCRPS}(F, y) = \int_{-\infty}^{\infty} w(x) (F(x) - \one{y \leq x} )^2 \dd x 
\end{align*}
and the CRPS is obtained in the special case, where $w$ is equal to one \citep{MathWink1976,GneitRanjan2011}. In order to emphasize the right tail, the choice $w(x) = \one{q \leq x}$ for some threshold $q \in \real$ can be used. Both wCRPS and CRPS are proper scoring rules as long as $\cF$ contains only distributions with finite first moments. In this case the CRPS is even strictly proper, while the wCRPS is only under additional assumptions, see \citet{GneitRaft2007}, \citet{GneitRanjan2011} and \citet{HolzKlar2018}.

As demonstrated by \citet[Section~2]{Tailletal2019}, the wCRPS is not able to clearly distinguish between different tail behavior. More precisely, given a distribution $G$ and $\varepsilon > 0$, it is always possible to construct a distribution $F$ that is not tail equivalent to $G$ and such that
\begin{align*}  
\vert \mathbb{E}\, \mathrm{wCRPS}(G, Y) - \mathbb{E}\, \mathrm{wCRPS} (F, Y) \vert \leq \varepsilon,
\end{align*}
where $Y$ has distribution $ G$. This results shows that for any distribution $G$ the tail can be modified while keeping the expected wCRPS $\varepsilon$-close to its minimum. As put by \citet{Tailletal2019} this means that the wCRPS is not a \textit{tail equivalent score}.

In the following we show that \emph{all} proper scoring rules fail to be tail equivalent in this sense. Moreover, we extend these findings to max-functionals, i.e.\ we show that no proper scoring rule is \textit{max-functional equivalent}. Both findings are immediate consequences of the subsequent continuity considerations for scoring rules.

\begin{definition}
A scoring rule $S : \cF \times \sO \rightarrow \real$ is called \textit{diagonal-continuous at} $G$ if for all $F \in \cF$
\begin{align*}
\bar{S}( \lambda F +  (1-\lambda)G, G) \rightarrow \bar{S}(G,G) \qquad \text{for } \lambda \downarrow 0.
\end{align*}
\end{definition}

\begin{lemma}  \label{lem:diagconti}
If $S : \cF \times \sO \rightarrow \real$ is a proper scoring rule, it is diagonal-continuous at each $G \in \cF$.
\end{lemma}

\begin{proof}
We proceed similar to the proof of \citet[Proposition~3]{Nau1985}. Let $F,G \in \cF$ and denote $F_\lambda := \lambda F + (1-\lambda) G$ for $\lambda \in [0,1)$. We obtain the inequality
\begin{align*}
(1-\lambda) \bar{S} (F_\lambda, G) &= \bar{S}(F_\lambda, F_\lambda) -  \lambda \bar{S}(F_\lambda, F) \\
&\leq \bar{S}(G,F_\lambda) - \lambda \bar{S}(F, F) \\
&= (1-\lambda) \bar{S}(G, G) + \lambda \big( \bar{S}(G,F) - \bar{S} (F, F) \big),
\end{align*}
since $S$ is a proper scoring rule. Rearranging leads to 
\begin{align*}
\vert \bar{S}( \lambda F  + (1-\lambda)G , G) - \bar{S}(G,G) \vert \leq \frac{\lambda}{1 - \lambda} \, \big( \bar{S}(G,F) - \bar{S} (F, F) \big),
\end{align*}
for $\lambda \in [0,1)$ and the right hand side of this equation vanishes as $\lambda \downarrow 0$.
\end{proof}

The argument of the proof of Lemma~\ref{lem:diagconti} can be extended to quasi-integrable scoring rules as considered in \citet{GneitRaft2007}. The additional requirement is that the expected score $\bar S(G,F)$ is finite and that $S$ is \emph{regular}, i.e.\ $\bar{S}(F,F) \in \real$ for all $F \in \cF$.

We can now turn our attention to the main result of this section. It is motivated by the observation that tail equivalence and max-functionals lead to a similar kind of discontinuity on  the convex combinations $\lambda F + (1-\lambda)G $, which intuitively conflicts with the diagonal-continuity of proper scoring rules. This allows for an extension of the results of \citet{Tailletal2019}. Recall the tail-ordering from Section~\ref{sec:examples} and that we assume $\cF$ to be convex.

\begin{theorem} \label{th:generalscoringdiff}
Let $S: \cF \times \real \rightarrow \real$ be a proper scoring rule and $G \in \cF$. Then the following are true.
\begin{enumerate}[label=(\alph*)]
	\item If there is an $F \in \cF$ with heavier tail than $G$, then for all $\varepsilon > 0$ there is an $F_\varepsilon \in \cF$ that is not tail equivalent to $G$ and such that
\begin{align*}
\vert \bar{S}(F_\varepsilon , G) - \bar{S}(G,G) \vert \leq \varepsilon .
\end{align*}
\item Let $T: \cF \rightarrow \real$ be a max-functional. If there is an $F \in \cF$ with $T(F) > T(G)$, then for all $\varepsilon > 0$ there is an $F_\varepsilon \in \cF$ such that $T(F_\varepsilon) =T(F) > T(G)$, while   
	\begin{align*}
\vert \bar{S}(F_\varepsilon , G) - \bar{S}(G,G) \vert \leq \varepsilon .
\end{align*}
\end{enumerate}
\end{theorem}

\begin{proof}
Fix $G \in \cF$ and let $S$ be a proper scoring rule. For $F \in \cF$ set $F_\lambda := \lambda F + (1-\lambda) G$. Since $\cF$ is convex, we have $ F_\lambda \in \cF$ for all $\lambda \in [0,1]$. Moreover, $S$ is diagonal-continuous at $G$ by Lemma~\ref{lem:diagconti}, implying that for all $\varepsilon > 0$ and $F \in \cF$ we can find a $\delta \in (0,1]$ such that $ \vert  \bar{S}( F_\lambda , G) - \bar{S}(G,G) \vert \leq \varepsilon$ holds for all $\lambda \in [0, \delta]$. Now assume there is an $F \in \cF$ with heavier tail than $G$. If $x^F > x^G$, we have $x^{F_\lambda} > x^G$ for all $\lambda \in (0,1]$. If on the other hand $x^F = x^G = x^*$ we have
\begin{align*}
\frac{\overline F_\lambda (x)}{ \overline G(x)}  = (1- \lambda) + \lambda \frac{\overline F(x)}{\overline G(x)} ,
\end{align*}
for $x < x^*$ and the right-hand side goes to infinity as $x \rightarrow x^*$. Hence, in both cases the distributions $F_\lambda$ cannot be tail equivalent to $G$ for $\lambda \in (0,1]$, showing part~(a). For the second part, let $F \in \cF$ satisfy $T(F) > T(G)$. Since $T$ is a max-functional, $T(F_\lambda) = T(F) > T(G)$ holds for $\lambda \in (0, 1]$, proving part~(b).
\end{proof}

The first part of Theorem~\ref{th:generalscoringdiff} shows that the lack of tail equivalence is not a flaw of the wCRPS, but inherent to \emph{all} proper scoring rules (up to integrability assumptions). The second part extends this non-equivalence of proper scoring rules to max-functionals. Loosely speaking, this means that there can not only be pairs of not tail equivalent distributions, but also pairs of distributions with arbitrarily different max-functional values, and both having almost identical expected scores.

\section{Discussion}  \label{sec:discuss}

Recent research investigates the elicitation properties of widely used statistical functionals. When the emphasis lies on an understanding of tail properties, typical functionals to characterize this behaviour fall into the class of max-functionals. In particular, all functionals that order distribution tails belong to this class (cf.~Proposition~\ref{prop:tailordering}). We show here that max-functionals do not only fail to be elicitable  (Theorem~\ref{th:Criterion2}), but have in fact infinite elicitation complexity in a wide range of settings (Theorem~\ref{th:complexity}). This contrasts situations in which the non-elicitability can be alleviated by a finite elicitation complexity as, for instance, is the case for the variance or the Expected Shortfall \citep{FronKash2018,FissZieg2016}.
Rather it bears resemblance to the mode, which is non-elicitable  and has infinite elicitation complexity as well, see \citet{Hein2014} and \citet{DearFron2019}. As an alternative to point forecasts, we may allow the max-functional to be reported via the entire distribution function. In principle such probabilistic forecasts can be compared using proper scoring rules. However, Theorem~\ref{th:generalscoringdiff} demonstrates that the difference of expected scores can be arbitrarily small, although the difference of  max-functional values may be large. The latter complements recent findings of \citet{Tailletal2019} and extends them from the wCRPS to all integrable proper scoring rules.

Collectively, our results cast doubt on the ability of expected scores to distinguish different tail regimes in the sense of max-functional values as they are routinely considered in extreme value theory. From an applied viewpoint this means that expected scores are not suitable to access such tail information for regression, M-estimation or comparative forecast evaluation. Thereby, our results provide a new perspective on the limitations of weighted scoring rules, adding to  practical intricacies described in \citet{Lerchetal2017}, \citet{HolzKlar2018} and \citet{FriedThor2012}. What might come to rescue though, is that the max-functionals themselves are often not the main concern in applications, but rather a tool to guide the extrapolation from intermediate order statistics to the functionals of interest. In practice, these functionals may include a high quantile or a tail expectation such as Expected Shortfall, which can be interpreted as tail properties `less extreme' than max-functionals and with better elicitablity properties. 

Lastly, however, we would like to point out that non-elicitability is not the only problem in sound forecast evaluation and many open questions remain. Even when elicitability is granted, i.e. the considered statistical functional is the unique minimizer of an expected score, there is no guarantee that the corresponding minimization problem will be well-posed. For instance, poorly behaved scoring functions may give rise to high variances of realized average scores, in which case practical sample sizes may be per se too low for an adequate assessment of competing forecasts. Due to the many challenges in forecast evaluation with an emphasis on distribution tails, we anticipate that it will remain an active area of research.

\section*{Acknowledgments}

Jonas Brehmer gratefully acknowledges support by the German Research Foundation (DFG) through the Research Training Group RTG 1953. 
The authors would also like to thank Tilmann Gneiting, Fabian Kr\"uger and Martin Schlather for their valuable comments.

\bibliographystyle{apalike}

\begin{thebibliography}{}

\bibitem[Bellini and Bignozzi, 2015]{BellBign2015}
Bellini, F. and Bignozzi, V. (2015).
\newblock On elicitable risk measures.
\newblock {\em Quantitative Finance}, 15:725--733.

\bibitem[Cadena and Kratz, 2016]{Cadena2016}
Cadena, M. and Kratz, M. (2016).
\newblock New results for tails of probability distributions according to their
  asymptotic decay.
\newblock {\em Statistics \& Probability Letters}, 109:178--183.

\bibitem[Dawid, 2007]{Dawid2007}
Dawid, A.~P. (2007).
\newblock The geometry of proper scoring rules.
\newblock {\em Annals of the Institute of Statistical Mathematics}, 59:77--93.

\bibitem[de~Haan and Ferreira, 2006]{FerrHaan2006}
de~Haan, L. and Ferreira, A. (2006).
\newblock {\em Extreme value theory}.
\newblock Springer Series in Operations Research and Financial Engineering.
  Springer, New York.

\bibitem[Dearborn and Frongillo, 2019]{DearFron2019}
Dearborn, K. and Frongillo, R. (2019).
\newblock On the indirect elicitability of the mode and modal interval.
\newblock {\em Annals of the Institute of Statistical Mathematics}.
\newblock To appear. Available at
  \url{https://doi.org/10.1007/s10463-019-00719-1}.

\bibitem[Diks et~al., 2011]{Diksetal2011}
Diks, C., Panchenko, V., and van Dijk, D. (2011).
\newblock Likelihood-based scoring rules for comparing density forecasts in
  tails.
\newblock {\em Journal of Econometrics}, 163:215--230.

\bibitem[Emmer et~al., 2015]{Emmetal2015}
Emmer, S., Kratz, M., and Tasche, D. (2015).
\newblock What is the best risk measure in practice? {A} comparison of standard
  measures.
\newblock {\em Journal of Risk}, 18:31--60.

\bibitem[Ferro and Stephenson, 2011]{FerroSteph2011}
Ferro, C. A.~T. and Stephenson, D.~B. (2011).
\newblock Extremal dependence indices: {I}mproved verification measures for
  deterministic forecasts of rare binary events.
\newblock {\em Weather and Forecasting}, 26:699--713.

\bibitem[Fissler et~al., 2019]{Fissetal2019}
Fissler, T., Hlavinov\'a, J., and Rudloff, B. (2019).
\newblock Elicitability and identifiability of systemic risk measures and other
  set-valued functionals.
\newblock Available at \url{https://arxiv.org/pdf/1907.01306.pdf}.

\bibitem[Fissler and Ziegel, 2015]{FissZieg2015}
Fissler, T. and Ziegel, J.~F. (2015).
\newblock Higher order elicitability and {O}sband's principle.
\newblock Available at \url{https://arxiv.org/pdf/1503.08123v3.pdf}.

\bibitem[Fissler and Ziegel, 2016]{FissZieg2016}
Fissler, T. and Ziegel, J.~F. (2016).
\newblock Higher order elicitability and {O}sband's principle.
\newblock {\em The Annals of Statistics}, 44:1680--1707.

\bibitem[Fissler and Ziegel, 2019]{FissZieg2019}
Fissler, T. and Ziegel, J.~F. (2019).
\newblock Order-sensitivity and equivariance of scoring functions.
\newblock {\em Electronic Journal of Statistics}, 13:1166--1211.

\bibitem[Friederichs and Thorarinsdottir, 2012]{FriedThor2012}
Friederichs, P. and Thorarinsdottir, T.~L. (2012).
\newblock Forecast verification for extreme value distributions with an
  application to probabilistic peak wind prediction.
\newblock {\em Environmetrics}, 23:579--594.

\bibitem[Frongillo and Kash, 2015]{FronKash2015b}
Frongillo, R. and Kash, I.~A. (2015).
\newblock Vector-valued property elicitation.
\newblock {\em Journal of Machine Learning Research: Workshop and Conference
  Proceedings}, 40:1--18.

\bibitem[Frongillo and Kash, 2018]{FronKash2018}
Frongillo, R. and Kash, I.~A. (2018).
\newblock Elicitation complexity of statistical properties.
\newblock Available at \url{https://arxiv.org/pdf/1506.07212.pdf}.

\bibitem[Gneiting, 2011]{Gneit2011}
Gneiting, T. (2011).
\newblock Making and evaluating point forecasts.
\newblock {\em Journal of the American Statistical Association}, 106:746--762.

\bibitem[Gneiting and Raftery, 2007]{GneitRaft2007}
Gneiting, T. and Raftery, A.~E. (2007).
\newblock Strictly proper scoring rules, prediction, and estimation.
\newblock {\em Journal of the American Statistical Association}, 102:359--378.

\bibitem[Gneiting and Ranjan, 2011]{GneitRanjan2011}
Gneiting, T. and Ranjan, R. (2011).
\newblock Comparing density forecasts using threshold- and quantile-weighted
  scoring rules.
\newblock {\em Journal of Business \& Economic Statistics}, 29:411--422.

\bibitem[Heinrich, 2014]{Hein2014}
Heinrich, C. (2014).
\newblock The mode functional is not elicitable.
\newblock {\em Biometrika}, 101:245--251.

\bibitem[Holzmann and Klar, 2017]{HolzKlar2018}
Holzmann, H. and Klar, B. (2017).
\newblock Focusing on regions of interest in forecast evaluation.
\newblock {\em The Annals of Applied Statistics}, 11:2404--2431.

\bibitem[Lambert et~al., 2008]{Lambetal2008}
Lambert, N.~S., Pennock, D.~M., and Shoham, Y. (2008).
\newblock Eliciting properties of probability distributions.
\newblock In {\em Proceedings of the 9th ACM Conference on Electronic
  Commerce}, EC '08, pages 129--138. ACM.

\bibitem[Ledford and Tawn, 1996]{LedTawn1996}
Ledford, A.~W. and Tawn, J.~A. (1996).
\newblock Statistics for near independence in multivariate extreme values.
\newblock {\em Biometrika}, 83:169--187.

\bibitem[Ledford and Tawn, 1997]{LedTawn1997}
Ledford, A.~W. and Tawn, J.~A. (1997).
\newblock Modelling dependence within joint tail regions.
\newblock {\em Journal of the Royal Statistical Society. Series B.
  Methodological}, 59:475--499.

\bibitem[Lerch et~al., 2017]{Lerchetal2017}
Lerch, S., Thorarinsdottir, T.~L., Ravazzolo, F., and Gneiting, T. (2017).
\newblock Forecaster's dilemma: extreme events and forecast evaluation.
\newblock {\em Statistical Science. A Review Journal of the Institute of
  Mathematical Statistics}, 32:106--127.

\bibitem[Matheson and Winkler, 1976]{MathWink1976}
Matheson, J.~E. and Winkler, R.~L. (1976).
\newblock Scoring rules for continuous probability distributions.
\newblock {\em Management Science}, 22:1087--1096.

\bibitem[McNeil et~al., 2015]{McNeiletal2015}
McNeil, A.~J., Frey, R., and Embrechts, P. (2015).
\newblock {\em Quantitative risk management}.
\newblock Princeton Series in Finance. Princeton University Press, Princeton,
  NJ, revised edition.

\bibitem[Nau, 1985]{Nau1985}
Nau, R.~F. (1985).
\newblock Should scoring rules be `effective'?
\newblock {\em Management Science}, 31:527--535.

\bibitem[Osband, 1985]{Osband1985}
Osband, K. (1985).
\newblock {\em Providing Incentives for Better Cost Forecasting}.
\newblock PhD thesis, University of California, Berkely.

\bibitem[Resnick, 1987]{Resnick1987}
Resnick, S.~I. (1987).
\newblock {\em Extreme values, regular variation, and point processes},
  volume~4 of {\em Applied Probability. A Series of the Applied Probability
  Trust}.
\newblock Springer-Verlag, New York.

\bibitem[Steinwart et~al., 2014]{Steinetal2014}
Steinwart, I., Pasin, C., Williamson, R., and Zhang, S. (2014).
\newblock Elicitation and identification of properties.
\newblock {\em Journal of Machine Learning Research: Workshop and Conference
  Proceedings}, 35:1--45.

\bibitem[Stephenson et~al., 2008]{Stephetal2008}
Stephenson, D.~B., Casati, B., Ferro, C. A.~T., and Wilson, C.~A. (2008).
\newblock The extreme dependency score: {A} non-vanishing measure for forecasts
  of rare events.
\newblock {\em Meteorological Applications}, 15:41--50.

\bibitem[Taillardat et~al., 2019]{Tailletal2019}
Taillardat, M., Foug{\`e}res, A.-L., Naveau, P., and De~Fondeville, R. (2019).
\newblock {Extreme events evaluation using CRPS distributions}.
\newblock Available at
  \url{https://hal.archives-ouvertes.fr/hal-02121796/file/CRPS-190429.pdf}.

\bibitem[Weber, 2006]{Weber2006}
Weber, S. (2006).
\newblock Distribution-invariant risk measures, information, and dynamic
  consistency.
\newblock {\em Mathematical Finance}, 16:419--441.

\end{thebibliography}

\end{document}